\documentclass[11pt]{article}
\usepackage{amsfonts}
\usepackage{amsmath,amssymb}

\usepackage{anysize} 
\marginsize{3cm}{3cm}{2cm}{2cm} 

\allowdisplaybreaks
\newtheorem{theorem}{Theorem}

\newtheorem{corollary}[theorem]{Corollary}

\newtheorem{lemma}[theorem]{Lemma}

\newtheorem{proposition}[theorem]{Proposition}

\newenvironment{proof}[1][Proof]{\noindent\textbf{#1.} }{\ \rule{0.5em}{0.5em}}

\makeatletter
\newcommand{\opnorm}{\@ifstar\@opnorms\@opnorm}
\newcommand{\@opnorms}[1]{%
  \left|\mkern-1.5mu\left|\mkern-1.5mu\left|
   #1
  \right|\mkern-1.5mu\right|\mkern-1.5mu\right|
}
\newcommand{\@opnorm}[2][]{%
  \mathopen{#1|\mkern-1.5mu#1|\mkern-1.5mu#1|}
  #2
  \mathclose{#1|\mkern-1.5mu#1|\mkern-1.5mu#1|}
}
\makeatother

\begin{document}

\title{\textbf{Blow-up at space infinity for solutions of a system of  non-autonomous semilinear heat equations}}
\author{{\textbf {Gabriela de Jes\'us Cabral-Garc\'ia}}\\
Universidad Aut\'onoma de Aguascalientes\\
Departamento de Matem\'{a}ticas y F\'isica\\
Aguascalientes, Aguascalientes, M\'exico\\
\ttfamily {gaby.cabral905@gmail.com}
\bigskip\\
{\textbf {Jos\'e Villa-Morales}}\\
Universidad Aut\'onoma de Aguascalientes\\
Departamento de Matem\'aticas y F\'isica\\
Aguascalientes, Aguascalientes, M\'exico\\
\ttfamily {jvilla@correo.uaa.mx} }
\date{}
\maketitle

\begin{abstract}
In this paper we will see that the global or local existence of solutions to
\begin{eqnarray*}
\dfrac{\partial u_{1}}{\partial t} & = & \mathit{k}_{1} (t) \Delta u_{1} + h_{1}(t) u_{1}^{p_{11}} u_{2}^{p_{12}},\\
\dfrac{\partial u_{2}}{\partial t} & = & \mathit{k}_{2} (t) \Delta u_{2} + h_{2}(t) u_{2}^{p_{22}} u_{1}^{p_{21}},
\end{eqnarray*}
depends on the initial datums and the global or local existence of solutions to 
\begin{eqnarray*}
\dfrac{dy_{1}}{dt} & = & h_{1}(t) y_{1}^{p_{11}}(t) y_{2}^{p_{12}}(t),\\
\dfrac{dy_{2}}{dt} & = & h_{2}(t) y_{2}^{p_{22}}(t) y_{1}^{p_{21}}(t).
\end{eqnarray*}
We also give some bounds for the maximal existence time of the partial differential system. Moreover, if such existence time is finite and $\min\{p_{11} + p_{12},p_{22}+p_{21}\} > 1$ then we will prove the partial differential system has solutions that blows-up at space infinite.

\vspace{0.5cm} \textbf{Mathematics Subject Classification (2010).} Primary
35K57, 35K45; Secondary 35B40, 35K20.

\end{abstract}

\section{Introduction and statement of main results}\label{SecPre}

Let $i \in \lbrace 1, 2 \rbrace $ and $j = 3-i$. We shall consider positive solutions of the following non-linear reaction-diffusion system
\begin{eqnarray}\label{eedif}
\dfrac{\partial u_{i}}{\partial t} (t,x) &=& \mathit{k}_{i} (t) \Delta u_{i} (t,x) + h_{i} (t) u_{i}^{p_{ii}} (t,x)  u_{j}^{p_{ij}} (t,x), \;\;\; t>0, \, x \in \mathbb{R}^{n}, \\
u_{i} (0,x) &=& \varphi_{i} (x), \;\;\; x \in \mathbb{R}^{n}, \nonumber
\end{eqnarray}
where $ \varphi_{i} : \mathbb{R}^{n} \rightarrow \mathbb{R} $ are continuous functions, non-negative and bounded, $ h_{i}, \, k_{i}: [0,\infty) \rightarrow (0,\infty)$ are continuous functions and $p_{ij}$ are non-negative real numbers, $i,j\in\{1,2\}$.

Let us introduce the following convention. The numbers $i$ and $j$ are dummy variables in the sense that if we define (or obtain) an expression for $i$, then we get a similar expression for $j$ changing only the roles of the indexes.

If $f$ is a real-valued function in the variables $(t,x)$, for $t$ fixed we will denote by $f(t,\cdot)$ the function $x\mapsto f(t,x)$, or briefly $f(t)=f(t,\cdot)$. Let $B(\mathbb{R}^{n})$ be the space of real-valued bounded measurable functions defined on $\mathbb{R}^{n}$. Let us consider the family $\{T_{t}, t\geq 0\}$ of bounded linear operators defined, on $B(\mathbb{R}^{n})$, as
\begin{equation*} 
T_{t} (g)(x) = (p(t,\cdot)*g)(x):=\int_{\mathbb{R}^{n}}p(t,x-y)g(y)dy,
\end{equation*}
where
\begin{equation*}
p (t,x) = \dfrac{1}{(4 \pi t)^{n/2}} \exp \left( - \dfrac{ \Vert x \Vert^{2}}{4t} \right),
\end{equation*}
is the Gaussian density. It is well known that $\{T_{t}, t\geq 0\}$ is a strongly continuous semigroup. With this notation, the corresponding integral system to (\ref{eedif}) is
\begin{equation} \label{eeint}
u_{i} (t,x) = U_{i}(0,t)(\varphi_{i})(x) +\int_{0}^{t} h_{i}(s) U_{i}(s,t) (u_{i}^{p_{ii}} (s) u_{j}^{p_{ij}} (s))(x)ds,
\end{equation} 
where $U_{i}(s,t)=T_{K_{i}(s,t)}$ and 
\begin{equation*}
K_{i} (s,t) := \int_{s}^{t} \mathit{k}_{i}(r)dr, 
\end{equation*}
in particular, we denote $K_{i} (0,t)$ by $K_{i} (t)$. A solution of (\ref{eeint}) is called mild solution of (\ref{eedif}), see \cite{Pazy}. The fact that (\ref{eedif}) is not an autonomous equation implies the family $\{U_{i}(s,t), 0\leq s<t\}$ is an evolution system (when $k_{i}\equiv 1$ then $U_{i}$ is a semigroup), see \cite{Pazy}. When we study the path behavior of mild solutions we usually use the semigroup property, but now we have an evolution system. This generality leads to some new difficulties that we are going to analysis in this paper.

We introduce a little more notation. Generically, by $[0,\tau_{(\cdot)})$ we will denote the maximal interval of existence of solution of the system $(\cdot)$. We say that (\ref{eedif}) has a global solution if $\tau_{(\ref{eedif})}=\infty$. On the contrary, if $\tau_{(\ref{eedif})}<\infty$ then (see Theorem 6.2.2 in \cite{Pazy})
\begin{equation}\label{liinfc}
\limsup_{t\uparrow\tau_{(\ref{eedif})}}\ \sup_{x\in\mathbb{R}^{n}}\{|u_{1}(t,x)|+|u_{2}(t,x)|\}=\infty
\end{equation}
and we will say the solution of (\ref{eedif}) blows-up in a finite time. Since we will study explosion in finite time we only need to verify that our system of equations of interest (\ref{eedif}) only has a local solution, in principle we do not need to show that it has a global solution. In fact, we prove, in Theorem \ref{extlocint}, that (\ref{eeint}) has a local solution, from this we can deduce that the system (\ref{eedif}) has a local solution.

Roughly speaking the diffusion term, in the system of equations (\ref{eedif}), just dispersed the mass of the system and the reaction term gives to the system a drift. Under certain hypotheses, we will see that to decide if $\tau_{(\ref{eedif})}$ is finite or not the diffusion term is not as important as the reaction term. To make precise this intuitive fact, let us consider the solution $(y_{1},y_{2})$ of the system
\begin{align} \label{edoa}
\frac{dy_{i}}{dt} (t) = & \ h_{i}(t)y_{i}^{p_{ii}}(t)y_{j}^{p_{ij}}(t), \ \ t>0,\\
y_{i}(0) = & \ ||\varphi_{i}||_{u},\nonumber
\end{align}
where $||f||_{u}=\sup\{|f(x)|:x\in \mathbb{R}^{n}\}$ is the uniform norm. 

\begin{theorem} \label{tdte}
For each $ i,j \in \left\lbrace 1,2 \right\rbrace $, let $ \varphi_{i} $ be a non-negative bounded continuous functions defined on $\mathbb{R}^{n}$ satisfying 
\begin{equation}
\lim \limits_{||x|| \rightarrow \infty} \varphi_{i} (x) = \Vert \varphi_{i} \Vert_{u}>0, \label{condleninf}
\end{equation}
$ h_{i}, \, k_{i}: [0,\infty) \rightarrow (0,\infty)$ are continuous functions and the constants $p_{ij}\geq 0$. Then $\tau_{(\ref{edoa})}=\tau_{(\ref{eedif})}$ and moreover we have
\begin{equation}\label{eqparalim}
\lim \limits_{||x|| \rightarrow \infty} u_{i} (t,x) = y_{i}(t),
\end{equation}
The above convergence is uniform on compact subset of $ [0,\tau_{(\ref{edoa})}) $.
\end{theorem}

The previous result means that the condition (\ref{condleninf}) together the system of equations (\ref{edoa}) determine completely the existence of global or local solutions to the system (\ref{eedif}). Such remarkable theorem was first proved in \cite{G-U} for the equation
\begin{eqnarray}
\dfrac{\partial u}{\partial t}(t,x) & = & \Delta u(t,x) + u^{p}(t,x), \ x\in \mathbb{R}^{n}, \ t>0,\label{edppot}\\
u(0,x)&=&u_{0}(x), \ x\in \mathbb{R}^{n}, \nonumber
\end{eqnarray}
where $p>1$ and the initial datum $u_{0}$ is a non-negative continuous function defined on $\mathbb{R}^{n}$ satisfying
\begin{equation}\label{condpexlim}
\lim_{||x||\rightarrow\infty}u_{0}(x)=M>0, \ \ u_{0}\leq M \ \ \text{and} \ \ u_{0}\not\equiv M.
\end{equation}
Now we omit the existence of some constant $M$, this is because it is not difficult to see that condition (\ref{condpexlim}) is equivalent to $M=||u_{0}||_{u}$. Moreover, we present a direct proof of Theorem \ref{tdte} based on some simple properties of $\liminf$. On the other hand, if we relax the assumption (\ref{condleninf}) then (\ref{eqparalim}) is not necessary true. In fact, Shimojo relaxed the condition (\ref{condpexlim}) but under such new condition, see \cite{S}, the limit $\lim_{||x|| \rightarrow \infty} u(t,x)$ is not necessary the solution to 
\begin{eqnarray}
\cfrac{dy}{dt}(t) & = & y^{p}(t), \ t>0, \label{edopot}\\
y(0)&=&M.\nonumber
\end{eqnarray}
The study of the global profile for the systems like (\ref{eedif}) is a recent topic of great interest (see for example \cite{V-A}, \cite{G-U}, \cite{J-B}, \cite{S}, \cite{S-U} and the references therein).
    
We will say that a solution of the system $(\ref{eedif})$ blows-up in finite time at space infinity if the maximal existence time $\tau_{(\ref{eedif})}<\infty$ and for all $R>0$,
\begin{equation}\label{exptfenifn}
\limsup_{t\uparrow \tau_{(\ref{eedif})}}\sup_{||x||\leq R}\{|u_{1}(t,x)|+|u_{2}(t,x)|\}<\infty.
\end{equation}
The limit (\ref{exptfenifn}) implies that for $t$ close enough to $\tau_{(\ref{eedif})}$ the functions $u_{1}(t)$ or $u_{2}(t)$ are  bounded in any closed ball, but on the other hand from (\ref{liinfc}) we can deduce that $||u_{1}(\tau_{(\ref{eedif})},\cdot)||_{u}=\infty$ or $||u_{2}(\tau_{(\ref{eedif})},\cdot)||_{u}=\infty$, so $u_{1}(t)$ or $u_{2}(t)$ are unbounded (or ``infinite'') just at infinity ($R\approx \infty$).

\begin{theorem}\label{Thbenalin}
Let us assume the hypotheses of Theorem \ref{tdte} and
\begin{equation}
\min\{p_{11} + p_{12},p_{22}+p_{21}\} > 1.\label{relentre}
\end{equation}
If $\tau_{(\ref{eedif})}<\infty$, then the solution of the system (\ref{eedif}) blows-up in finite time at space infinity.
\end{theorem}

If, for each $i\in \{1,2\}$, we set $u_{i}(\tau_{(\ref{eedif})}-):=\limsup_{t\uparrow \tau_{(\ref{eedif})}}u_{i}(t)$, from the proof of Theorem \ref{Thbenalin} we will see that $|u_{i}(\tau_{(\ref{eedif})}-,x)|<\infty$, for each $x\in\mathbb{R}^{n}$ and $||u_{1}(\tau_{(\ref{eedif})}-)||_{u}+||u_{2}(\tau_{(\ref{eedif})}-)||_{u}=\infty$, this clarifies the name of blow-up in finite time at space infinity.

The condition (\ref{relentre}) could be interpreted as a strong cooperative relationship in the system in order to have this kind of explosion. A similar autonomous system is studied in \cite{S-U}, but now we consider the non-autonomous one and we observe that the time dependence, through $k$ and $h$, does not affect the path behavior of the solution. This is intuitively clear because the explosion in time is finite, and in this case the contribution of $k$ and $h$ is bounded, since they are continuous functions.

The explosion time of equation (\ref{edopot}) is $t_{e}=(p-1)^{-1}M^{1-p}$, $p>1$. Using this, the authors in \cite{G-U} proved that the maximal time of existence of (\ref{edppot}) is $t_{e}$. In our case, the system (\ref{edoa}) does not have an explicit expression. Therefore, one of our main contributions is to give bounds for the maximum time of existence of the system (\ref{eedif}). 
To state our next result we need to introduce some new notation. By $\tau_{(\cdot),i}$ we mean the maximal existence time of the $i$-th component of the system $(\cdot)$, then $\tau_{(\cdot)}=\min\{\tau_{(\cdot),1},\tau_{(\cdot),2}\}$.  We also set
\begin{equation*}
a_{i} = p_{ji} -p_{ii} +1, \ \ \ \alpha_{i} = p_{ii} +p_{ij} \dfrac{a_{i}}{a_{j}},
\end{equation*}
and
\begin{equation*}
\tilde{h}_{i}(t) = \dfrac{h_{i}(t)}{h_{j}(t)}.
\end{equation*}
In what follows, we denote by $c_{(\cdot)}$ a positive constant related to the inequality (or expression) indicated in $(\cdot)$. If $f:\mathbb{R} \rightarrow \mathbb{R}$ is monotone by $f(\infty)$ we denote $\lim_{x\rightarrow\infty}f(x)$. 

\begin{theorem}\label{thest}
Let us assume that $p_{ij}\geq 0$, $ h_{i}, \, k_{i}: [0,\infty) \rightarrow (0,\infty)$ are continuous functions and $(\tilde{h}_{i})'\leq 0$, for each $i,j\in \{1,2\}$. 
\begin{itemize}
\item[{$(a)$}] Suppose $ a_{i} > 0 $ and $ a_{j} \neq 0 $.
\begin{itemize}
\item[$(a.1)$] If $ \alpha_{j} > 1 $ and $ F_{1,j} (\infty) > y_{j}^{1-\alpha_{j}}(0)(\alpha_{j}-1)^{-1}$, then
\begin{equation*}
\tau_{(\ref{edoa}),j} \leq F_{1,j}^{-1} \left( \dfrac{y_{j}^{1-\alpha_{j}}(0)}{\alpha_{j}-1} \right),
\end{equation*}
where
\begin{equation*}
F_{1,j}(x) = c_{(\ref{epi})}^{p_{ji}} \int_{0}^{x} h_{j}(t) (\tilde{h}_{i}(t))^{p_{ji}/a_{i}}dt.
\end{equation*}
\item[$(a.2)$] If $ \alpha_{i} \leq 1 $ or $[ \alpha_{i} > 1 \ \text{and} \ F_{1,i} (\infty) \leq y_{i}^{1-\alpha_{i}} (0) / (\alpha_{i}-1) ] $, then $ \tau_{(\ref{edoa}),i} = \infty$, where
\begin{equation*}
F_{1,i}(x) = c_{(\ref{epi})}^{p_{ii}-\alpha_{i}} \int_{0}^{x} h_{i}(t) (\tilde{h}_{i}(t))^{-p_{ij}/a_{j}}dt.
\end{equation*}
\end{itemize}
\item[$(b)$] Suppose $ a_{i} > 0 $ and $ a_{j} = 0 $.
\begin{itemize}
\item[$(b.1)$] If $ y_{j} (0) > 1 $ and $ F_{2,j} (\infty) > B_{2,j} (\infty) $, then $\tau_{(\ref{edoa}),j} \leq F_{2,j}^{-1} (B_{2,j} (\infty) ) $, where
\begin{equation*}
F_{2,j}(x) =  c_{(\ref{epii})}^{p_{ji}} \int_{0}^{x} h_{j}(t) \tilde{h}_{i} ^{p_{ji}/a_{i}} (t)dt, \; \;
B_{2,j} (x) = \int_{y_{j}(0)}^{x} \dfrac{ds}{s^{p_{jj}} (\log s)^{p_{ji}/a_{i}} }.
\end{equation*}
\item[$(b.2)$] If $ \tilde{h}_{i} (t) \equiv \tilde{c}_{i}>0 $ and $ F_{2,i} (\infty) \leq B_{2,i} (\infty) $, then $\tau_{(\ref{edoa}),i} = \infty $, where
\begin{equation*}
F_{2,i}(x) = \int_{0}^{x} h_{i}(t) dt, \; \;
B_{2,i} (x) = \int_{y_{i}(0)}^{x} \dfrac{1}{s^{p_{ii}}}\exp\left(-p_{ij}c_{(\ref{epii})}^{-1}(\tilde{c}_{i})^{-1} s^{a_{i}}\right)ds.
\end{equation*}
\end{itemize}
\item[$(c)$] Suppose $ a_{i} = 0 $, $ a_{j} = 0 $, $ \tilde{h}_{i}(t)\equiv \tilde{c}_{i} > 0$.    
\begin{itemize}
\item[$(c.1)$] If $\beta_{j}:= p_{ii} + p_{ij} \tilde{c}_{i}>1$ and $ F_{3,j}(\infty) > y_{j}^{1-\beta_{j}} (0) / (\beta_{j}-1) $, then
\begin{equation*}
\tau_{(\ref{edoa}),j}  \leq F_{3,j}^{-1} \left(\dfrac{y_{j}^{1-\beta_{j}} (0)}{\beta_{j}-1} \right),
\end{equation*}
where
\begin{equation*}
F_{3,j}(x) = \left( \dfrac{y_{i}(0)}{y_{j}^{\tilde{c}_{i}} (0)} \right)^{p_{ij}} \int_{0}^{x} h_{j}(t) dt.
\end{equation*}
\item[$(c.2)$] If $\gamma_{i}:= p_{ii} + p_{ij}(\tilde{c}_{i})^{-1}\leq 1$ or $ [ \gamma_{i} > 1 \text{ and } F_{3,i} (\infty) \leq y_{i}^{1-\gamma_{i}} (0) / (\gamma_{i}-1) ] $, then $ \tau_{(\ref{edoa}),i}  = \infty $, where
\begin{equation*}
F_{3,i}(x) = \left( \dfrac{y_{j}(0)}{y_{i}^{1 / \tilde{c}_{i}} (0)} \right)^{p_{ij}} \int_{0}^{x} h_{i}(t) dt.
\end{equation*}
\end{itemize}
\end{itemize}
\end{theorem}

In \cite{Q-R} is studied the possibility of non-simultaneous blow-up for positive solutions of the system (\ref{eedif}), when $h_{i}\equiv1, \ k_{i}\equiv 1$, $i=1,2$. Now we are able to get different conditions for non-simultaneous blow-up. For example, suppose $h_{1}\equiv 1, \ h_{2}\equiv 1$ in (\ref{eedif}) and $a_{i}>0$, $a_{j}\neq 0$. If $\alpha_{j}>1$ and $\alpha_{i}\leq 1$, then Theorem \ref{thest}$(a)$ implies  that $\tau_{(\ref{edoa}),j}<\infty$ and $\tau_{(\ref{edoa}),i}=\infty$. Moreover, such result tell us intuitively that the simultaneous blow up ``is not very common.''

In the following result, which is a generalization of the blow-up criterion given in \cite{E-L} or \cite{S-U}, we will not impose any restriction on the functions $k_{i}$ but if over $h_{i}$ (see also Proposition \ref{wecandance}). This means, as we said before, that the reaction term is the important ingredient to determine the existence of global or local solutions.
\begin{corollary}\label{Corollary}
For $i\in\{1,2\}$, let us assume $ h_{i} \equiv \tilde{c}_{i}$ and let us take $k_{i}$, $\varphi_{i}$ as in Theorem \ref{tdte}. If  $ p_{11} > 1 $,  $ p_{22} > 1 $  or  $ (p_{11}-1)(p_{22}-1) - p_{12}p_{21} < 0 $, then $\tau_{(\ref{edoa})}<\infty$.
\end{corollary}

\bigskip
The paper is organized as follows. Using the Banach contraction principle we prove, in Section \ref{STheLocal}, that the system (\ref{eedif}) has a local solution. The Theorem \ref{tdte} is proved in Section \ref{Smaximal} and it follows from a comparison theorem for an integral system of equations, which we consider is important in itself. In Section \ref{SBlow} we give the proof of Theorem \ref{Thbenalin} adapting some ideas introduced in \cite{S-U}. Finally, using mainly a generalized version of Osgood's lemma we prove Theorem \ref{thest} in Section \ref{Sbounds}.

\section{Local existence}\label{STheLocal}
In this section we prove the existence of local solutions to the system (\ref{eedif}). The proof is standard and we only present the main ideas. We begin with the following elementary equality, which will be essential in some steps.
\begin{lemma} \label{estpp}
Let $a,b,c,d \geq 0$ and $p,q \geq 0$, then 
\begin{equation*}
\begin{split}
a^{p}b^{q} - c^{p}d^{q} = & p(a-c)\int_{0}^{1} (c+t(a-c))^{p-1}(d+t(b-d))^{q} dt \\
& + \, q(b-d)\int_{0}^{1} (c+t(a-c))^{p}(d+t(b-d))^{q-1}dt.
\end{split}
\end{equation*}
\end{lemma}

\begin{proof}
Considering the function $ (x,y) \mapsto x^{p}y^{q} $ and using the mean value theorem for several variables (see Exercise 4.W in Section 40 of \cite{B}) we get the desired equality.\hfill
\end{proof}

\bigskip
To deal with the existence of local solutions to the system (\ref{eedif}) let us consider the space $B([0,\mathit{T}] \times \mathbb{R}^{n}) $ of real-valued bounded measurable functions defined on $ [0,\mathit{T}] \times \mathbb{R}^{n} $, for some $ \mathit{T} > 0 $, that we are going to fix later. The set $E_{\mathit{T}} = B([0,\mathit{T}] \times \mathbb{R}^{n}) \times B([0,\mathit{T}] \times \mathbb{R}^{n}) $ is a Banach space with the norm
\begin{equation*}
\opnorm{(u_{1},u_{2})}  =  \sup \limits_{t \in [0,\mathit{T}]} \left\lbrace\Vert u_{1}(t,\cdot) \Vert_{u}+\Vert u_{2}(t,\cdot) \Vert_{u}\right\rbrace.
\end{equation*}
Let $\bar{B}_{\mathit{T}} (R) = \left\lbrace (u_{1},u_{2}) \in E_{\mathit{T}} : \opnorm{(u_{1},u_{2})} \leq R \right\rbrace $ be the closed ball in $E_{T}$ with center at $ (0,0) $ and radius $ R $  and $\mathit{P}_{\mathit{T}} = \lbrace (u_{1},u_{2}) \in E_{\mathit{T}}: u_{1} \geq 0, \; u_{2} \geq 0 \rbrace$. Since $A_{T}:=\bar{B}_{T}\cap P_{T}$ is a closed subset of $E_{T}$, then $A_{T}$ is also a Banach space.
\begin{theorem} \label{extlocint}
Under the hypotheses of Theorem \ref{tdte} there exists a $T>0$ such that (\ref{eedif}) has a unique positive solution $(u_{1},u_{2})$ in $C^{1,2}([0,T]\times \mathbb{R}^{n})\times C^{1,2}([0,T]\times \mathbb{R}^{n})$. 
\end{theorem}

\begin{proof}
Define the function $ \psi : E_{\mathit{T}} \rightarrow E_{\mathit{T}} $, as
\begin{eqnarray*}
\psi (u_{1},u_{2})(t,x)& = & \left((p(K_{1} (t)) \ast\varphi_{1}) (x), (p(K_{2} (t)) \ast\varphi_{2}) (x)\right) \, \\
&&+ \left( \int_{0}^{t} h_{1}(s) (p(K_{1} (s,t))\ast u_{1}^{p_{11}} (s) u_{2}^{p_{12}} (s))(x)ds, \right. \\
&& \left. \int_{0}^{t} h_{2}(s) (p(K_{2} (s,t))\ast u_{1}^{p_{21}} (s) u_{2}^{p_{22}} (s))(x)ds \right).
\end{eqnarray*}
By hypothesis $ \varphi_{1},\varphi_{2} \geq 0 $, then (\ref{eeint}) implies $(u_{1},u_{2}) \in \mathit{P}_{\mathit{T}} $.

The Gaussian density means that for each $t>0$
\begin{equation}\label{intp}
\int_{\mathbb{R}^{n}} p(t,y) dy = 1.
\end{equation}
If $(u_{1},u_{2}) \in \bar{B}_{\mathit{T}}(R)$, then (\ref{intp}) implies
\begin{eqnarray*}
&&\left| (p(K_{i}(t)) \ast\varphi_{i})(x)+ \int_{0}^{t} h_{i}(s) (p(K_{i}(s,t)) \ast u_{i}^{p_{ii}} (s) u_{j}^{p_{ij}} (s))(x)ds \right| \\
&&\leq \int_{\mathbb{R}^{n}} p(K_{i}(t),y-x) \Vert \varphi_{i} \Vert_{u} dy \\
&& \ \ \ + \int_{0}^{t} h_{i}(s) \int_{\mathbb{R}^{n}} p(K_{i}(s,t),y-x) \Vert u_{i} (s) \Vert_{u}^{p_{ii}} \Vert u_{j} (s) \Vert_{u}^{p_{ij}} dy ds \\
&& =\Vert \varphi_{i} \Vert_{u} + \int_{0}^{t} h_{i}(s) \Vert u_{i} (s) \Vert_{u}^{p_{ii}} \Vert u_{j} (s) \Vert_{u}^{p_{ij}} ds \\
&&\leq  \Vert \varphi_{i} \Vert_{u} + \mathit{R}^{p_{ii} + p_{ij}} \int_{0}^{t} h_{i}(s) ds.
\end{eqnarray*}
From this we get
\begin{equation*}
\opnorm{\psi(u_{1},u_{2})} \leq \max \left\lbrace \Vert \varphi_{1} \Vert_{u} + \mathit{R}^{p_{11} + p_{12}} \int_{0}^{t} h_{1}(s)ds, \, \Vert \varphi_{2} \Vert_{u} + \mathit{R}^{p_{21} + p_{22}} \int_{0}^{t} h_{2}(s) ds \right\rbrace.  
\end{equation*}
Let us take
\begin{equation*}
\mathit{R} = \max \left\lbrace \Vert \varphi_{1} \Vert_{u} + 1, \Vert \varphi_{2} \Vert_{u} + 1 \right\rbrace,
\end{equation*}
and some $ \mathit{T} > 0 $ small enough such that
\begin{equation*}
\max \left\lbrace \mathit{R}^{p_{11} + p_{12}} \int_{0}^{\mathit{T}} h_{1}(s) ds, \, \mathit{R}^{p_{21} + p_{22}} \int_{0}^{\mathit{T}} h_{2}(s) ds \right\rbrace < 1.
\end{equation*}
This implies that $ \psi(u_{1},u_{2}) \in \bar{B}_{\mathit{T}}(R)$. Therefore  $\psi: A_{\mathit{T}} \rightarrow A_{\mathit{T}} $.

Now let us see that $ \psi $ is a contraction. Let us take $ (u_{1},u_{2}), (v_{1},v_{2}) \in A_{T}$.  Lemma \ref{estpp} and (\ref{intp}) turns out
\begin{eqnarray*}
&& \left| \int_{0}^{t} h_{i}(s) \int_{\mathbb{R}^{n}} p(K_{i}(s,t),y-x) u_{i}^{p_{ii}} (s,y) u_{j}^{p_{ij}} (s,y) dy ds  \right. \\
&& \ \ \ -\left. \int_{0}^{t} h_{i}(s) \int_{\mathbb{R}^{n}} p(K_{i}(s,t),y-x) v_{i}^{p_{ii}} (s,y) v_{j}^{p_{ij}} (s,y) dy ds \right| \\
&&\leq  \int_{0}^{t} h_{i}(s) \int_{\mathbb{R}^{n}} p(K_{i}(s,t),y-x) \left| u_{i}^{p_{ii}} (s,y) u_{j}^{p_{ij}} (s,y) - v_{i}^{p_{ii}} (s,y) v_{j}^{p_{ij}} (s,y) \right| dy ds \\
&&\leq  \dfrac{(3\mathit{R})^{p_{ii}+p_{ij}}}{2 \mathit{R} (p_{ii}+p_{ij})} \int_{0}^{t} h_{i}(s)(p_{ii} \Vert u_{i} (s) - v_{i} (s) \Vert_{u} + p_{ij} \Vert u_{j} (s) - v_{j} (s) \Vert_{u}) ds \\
&&\leq   \left( (3\mathit{R})^{p_{ii}+p_{ij}}\int_{0}^{t} h_{i}(s) ds \right) \opnorm{u - v}.
\end{eqnarray*}
Taking the supreme in $t$, in the above inequality, we get
\begin{eqnarray*}
&&\opnorm{\psi(u_{1},u_{2}) - \psi(v_{1},v_{2})} \\
&&\leq  \max \left\lbrace (3\mathit{R})^{p_{11}+p_{12}} \int_{0}^{\mathit{T}} h_{1}(s) ds,(3\mathit{R})^{p_{21}+p_{22}} \int_{0}^{\mathit{T}} h_{2}(s) ds \right\rbrace\opnorm{u- v}.
\end{eqnarray*}
Taking $ \mathit{T} > 0 $ small enough we see that $ \psi $ is a contraction, then the Banach contraction principle implies that there exists a unique $(u_{1},u_{2})\in A_{T}$ such that, for each $i,j \in\{1,2\}$,
\begin{equation}
u_{i}(t,x) = (p(K_{i} (t)) \ast\varphi_{i}) (x)+ \int_{0}^{t} h_{i}(s) (p(K_{i} (s,t))\ast u_{i}^{p_{ii}} (s) u_{j}^{p_{ij}} (s))(x)ds. \label{ecintexp}
\end{equation}

This means that $(u_{1},u_{2})$ is a mild local solution for the system (\ref{eedif}). From the basic properties of the convolution operator, $\ast$, and using the Lebesgue dominated convergence theorem we can see that $(u_{1},u_{2})\in C^{1,2}([0,T]\times \mathbb{R}^{n})\times C^{1,2}([0,T]\times \mathbb{R}^{n})$ and satisfies the system (\ref{eedif}).\hfill
\end{proof}

\section{A characterization of the maximal existence time}\label{Smaximal}

Remember that  $\tau_{(\ref{eedif})}$ is the maximal existence time for the systems (\ref{eedif}), the above result implies $\tau_{(\ref{eedif})}\geq T>0$.
 
\begin{theorem} \label{lcpo}
Let $ z_{1} $ and $ z_{2} $ be non-negative measurable functions defined on $[0,\tau_{(\ref{edoa})})$and $\tau \in (0,\tau_{(\ref{edoa})}]$. Suppose that, for each $i,j\in\{1,2\}$,
\begin{equation*}
z_{i} (t) \geq(\leq) \ z_{i} (0) + \int_{0}^{t} h_{i} (s) z_{i}^{p_{ii}} (s) z_{j}^{p_{ij}} (s) ds, \;\;\; \text{for all } t \in [0,\tau). 
\end{equation*}
If $ z_{i}(0)= y_{i}(0) $, then
\begin{equation} \label{udfg}
z_{i} (t) \geq (\leq) \ y_{i} (t), \;\;\; \text{for all } t \in [0,\tau).
\end{equation}
\end{theorem}

\begin{proof}
We only work with the case
\begin{equation*}
z_{i} (t) \geq z_{i} (0) + \int_{0}^{t} h_{i} (s) z_{i}^{p_{ii}} (s) z_{j}^{p_{ij}} (s) ds, \; \forall t \in [0,\tau),
\end{equation*}
the other case is similar. Let us introduce the sets 
\begin{equation*}
A_{i}=\left\lbrace t \in[0,\tau): z_{i}(s) \geq y_{i} (s), \; \text{for all } s \in [0,t] \right\rbrace.
\end{equation*}
Observe that $ A_i \neq \emptyset \; (0 \in A_{i}) $, then $ \mathit{T}_{i} = \sup A_{i} \leq \tau $ is well defined. We will see that $ \mathit{T}_{i} = \tau$, $i\in \{1,2\}$. Let us proceed by contradiction, suppose $ \min\{T_{1},T_{2}\} < \tau$. Without loss of generality let us assume that $T_{1}\leq T_{2}<\tau$. Using Lemma \ref{estpp} we have, for each $ t\in [0,\tau-\mathit{T}_{1}) $,
\begin{eqnarray}
z_{1}(\mathit{T}_{1}+t) - y_{1}(\mathit{T}_{1}+t) &\geq& \int_{0}^{\mathit{T}_{1}+t} h_{1}(s) \left[ z_{1}^{p_{11}}(s) z_{2}^{p_{12}} (s) - y_{1}^{p_{11}} (s) y_{2}^{p_{12}} (s) \right] ds \nonumber\\
 &=& \int_{0}^{\mathit{T}_{1}+t} h_{11}(s) \left( z_{1}(s) - y_{1}(s) \right) ds \nonumber\\
&& + \int_{0}^{\mathit{T}_{1}+t} h_{12}(s) \left( z_{2}(s) - y_{2}(s) \right) ds,\label{eccparados}
\end{eqnarray}
where
\begin{align*}
h_{11}(s) &= p_{11}h_{1} (s) \int_{0}^{1} \left( y_{1}(s) + \xi (z_{1}(s)-y_{1}(s)) \right) ^{p_{11}-1} \left( y_{2}(s) + \xi (z_{2}(s)-y_{2}(s)) \right) ^{p_{12}} d \xi, \\ 
h_{12}(s) &= p_{12}h_{1} (s) \int_{0}^{1} \left( y_{1}(s) + \xi (z_{1}(s)-y_{1}(s)) \right) ^{p_{11}} \left( y_{2}(s) + \xi (z_{2}(s)-y_{2}(s)) \right) ^{p_{12}-1} d \xi.
\end{align*}
Notice that 
$$y_{i}(s)\geq y_{i}(0)>0, \ \ z_{i}(s)\geq z_{i}(0)>0,$$
then for each $\xi\in[0,1]$
\begin{align}
y_{1}(s) + \xi (z_{1}(s)-y_{1}(s))=\xi z_{1}(s)+(1-\xi)y_{1}(s)>0, \nonumber\\
y_{2}(s) + \xi (z_{2}(s)-y_{2}(s))=\xi z_{2}(s)+(1-\xi)y_{2}(s)>0. \label{estpapo}
\end{align}
From this we can see that $h_{11}\geq 0$ and $h_{12}\geq 0$.

Let $ \varepsilon > 0 $ and define $\psi_{i\varepsilon}: [0,\tau - \mathit{T}_{1})\rightarrow \mathbb{R}$, by
\begin{equation*}
\psi_{i\varepsilon} (t) = z_{i}(\mathit{T}_{1}+t) - y_{i}(\mathit{T}_{1}+t) + \varepsilon,
\end{equation*}
and $ k_{ij} (t) = h_{ij} (\mathit{T}_{1} + t) $.
Then (\ref{eccparados}) leads to
\begin{equation*}
\psi_{1\varepsilon} (t) \geq \int_{0}^{t} k_{11}(s) \psi_{1\varepsilon}(s) ds+ \int_{0}^{t} k_{12}(s) \psi_{2\varepsilon}(s) ds+ \varepsilon \left( 1 - \int_{0}^{t} [k_{11}(s) +k_{12}(s)]ds \right),
\end{equation*}
analogously
\begin{equation*}
\psi_{2\varepsilon} (t) \geq \int_{0}^{t} k_{22}(s) \psi_{2\varepsilon}(s) ds+ \int_{0}^{t} k_{21}(s) \psi_{1\varepsilon}(s) ds+ \varepsilon \left( 1 - \int_{0}^{t} [k_{21}(s) +k_{22}(s)]ds \right).
\end{equation*}

The hypothesis $h_{i}>0$ and (\ref{estpapo}) implies $k:=k_{11}+k_{12}+k_{21}+k_{22}>0$ on $[0,\tau-T_{1}]$, therefore $\sup\{k(t):t\in [0,(\tau-T_{1}])/2]\}\in(0,\infty)$, then there exits $\tilde{t}>0$ such that $\int_{0}^{\tilde{t}}k(s)ds<1/2$. Let us define
\begin{equation*}
\tilde{\mathit{T}}_{1} := \sup \left\lbrace t > 0: \int_{0}^{t} k(s)ds < \dfrac{1}{2} \right\rbrace \wedge  \dfrac{\tau-\mathit{T}_{1}}{2},
\end{equation*}
then $\tilde{\mathit{T}}_{1}\geq \tilde{t} \wedge  (\tau-\mathit{T}_{1})/2>0$. From the definition of $\tilde{\mathit{T}}_{1}$ we get, for each $t \in [0,\tilde{T}_{1})$,
\begin{eqnarray}
\psi_{1\varepsilon} (t) &\geq& \int_{0}^{t} k_{11}(s) \psi_{1\varepsilon}(s) ds+ \int_{0}^{t} k_{12}(s) \psi_{2\varepsilon}(s) ds+ \frac{\varepsilon}{2}, \label{mesrija}\\
\psi_{2\varepsilon} (t) &\geq& \int_{0}^{t} k_{22}(s) \psi_{2\varepsilon}(s) ds+ \int_{0}^{t} k_{21}(s) \psi_{1\varepsilon}(s) ds+ \frac{\varepsilon}{2}.\nonumber
\end{eqnarray}

Now let us introduce the set 
\begin{equation*}
B_{i}=\{t \in[0,\tilde{T}_{1}): \psi_{i\varepsilon} (s) \geq 0, \; \text{for all } s \in [0,t]\}.
\end{equation*}
The sets $B_{i}$ are not empty because $\psi_{i\varepsilon} (0) \geq \varepsilon/2$. Set $\mathit{S}_{i} = \sup B_{i} \leq \tilde{T}_{1}$. We will see that $ \mathit{S}_{i} = \tilde{T}_{1}$. Let us proceed by contradiction, $ \min\{S_{1},S_{2}\} < \tilde{T}_{1}$. Without loss of generality let us suppose $S_{1}\leq S_{2}<\tilde{T}_{1}$. The assumption $S_{1}\leq S_{2}$ implies
\begin{align*}
\lim_{t\rightarrow 0}\left\{\int_{0}^{S_{1}+t} k_{11}(s) \psi_{1\varepsilon}(s) ds+ \int_{0}^{S_{1}+t} k_{12}(s) \psi_{2\varepsilon}(s) ds+ \frac{\varepsilon}{2}\right\}\\
=\int_{0}^{S_{1}} k_{11}(s) \psi_{1\varepsilon}(s) ds+ \int_{0}^{S_{1}} k_{12}(s) \psi_{2\varepsilon}(s) ds+ \frac{\varepsilon}{2}\geq\frac{\varepsilon}{2},
\end{align*}
then there is a $\delta_{1}>0$, such that
\begin{eqnarray*}
\psi_{1\varepsilon} (S_{1}+t)&\geq&\int_{0}^{S_{1}+t} k_{11}(s) \psi_{1\varepsilon}(s) ds+ \int_{0}^{S_{1}+t} k_{12}(s) \psi_{2\varepsilon}(s) ds+ \frac{\varepsilon}{2}\\
&\geq& \frac{\varepsilon}{4}, \ \ \text{ for all } \ t\in [0,\delta_{1}].
\end{eqnarray*}
Hence $S_{1}+\delta_{1}\in B_{1}$, contradicting the definition of $S_{1}$. 

So we have seen that $\psi_{i\varepsilon} (t)\geq 0$, for each $t \in [0,\tilde{T}_{1})$. In this way, the inequality (\ref{mesrija}) yields
$$\psi_{1\varepsilon} (t) \geq \frac{\varepsilon}{2},\ \ \text{for all } t \in [0,\tilde{T}_{1}),$$
therefore
\begin{equation*}
z_{1}(\mathit{T}_{1}+t) - y_{1}(\mathit{T}_{1}+t) \geq - \frac{\varepsilon}{2}, \; \; \text{for all } t \in [0,\tilde{\mathit{T}_{1}}).
\end{equation*}
Observing that $ \tilde{\mathit{T}_{1}} $ does not depend on $\varepsilon$ and letting $ \varepsilon \rightarrow 0$ we have
\begin{equation*}
z_{1}(t) \geq y_{1}(t),  \;\;\; \text{for all } t \in [0,\mathit{T}_{1} + \tilde{\mathit{T}_{1}}).
\end{equation*}
Due to $ \tilde{\mathit{T}_{1}} > 0 $, we get a contradiction to the definition of $ \mathit{T}_{1} $. Obtaining the desired result, $ \mathit{T}_{1} = \mathit{T}_{2}=\tau$.\hfill
\end{proof}

\begin{proof}[Proof of Theorem \ref{tdte}]
Taking $\liminf$ on both sides of the inequality (\ref{ecintexp}) we have
\begin{align*}
\liminf\limits_{||x|| \rightarrow \infty} u_{i}(t,x) \geq & \liminf\limits_{||x|| \rightarrow \infty} \int_{\mathbb{R}^{n}} p(K_{i}(t),y) \varphi_{i} (y-x) dy \\
& + \liminf\limits_{||x|| \rightarrow \infty} \int_{0}^{t} h_{i} (s) \int_{\mathbb{R}^{n}} p(\mathit{K}_{i}(s,t),y) u_{i}^{p_{ii}}(s,y-x) u_{j}^{p_{ij}} (s,y-x) dy ds.
\end{align*}
Fatou's lemma yields
\begin{align*}
\liminf\limits_{||x|| \rightarrow \infty} u_{i}(t,x) \geq & \int_{\mathbb{R}^{n}} p(K_{i}(t),y) \liminf\limits_{||x|| \rightarrow \infty} \varphi_{i} (y-x) dy \\
& + \int_{0}^{t} h_{i} (s) \int_{\mathbb{R}^{n}} p(\mathit{K}_{i}(s,t),y) \liminf\limits_{||x|| \rightarrow \infty} u_{i}^{p_{ii}}(s,y-x) u_{j}^{p_{ij}} (s,y-x) dy ds \\
\geq & \liminf\limits_{||x|| \rightarrow \infty} \varphi_{i} (x) + \int_{0}^{t} h_{i} (s) \int_{\mathbb{R}^{n}} p(\mathit{K}_{i}(s,t),y) \left(  \liminf\limits_{||x|| \rightarrow \infty} u_{i}(s,x) \right)^{p_{ii}}  \\
& \times \left( \liminf\limits_{||x|| \rightarrow \infty} u_{j} (s,x)\right) ^{p_{ij}} dy ds. 
\end{align*}
Given $ y \in \mathbb{R}^{n} $ the limits on the right hand side of the above inequality does not depend on $ y $, so we can use (\ref{intp}) to get
\begin{equation}
\liminf\limits_{||x|| \rightarrow \infty} u_{i}(t,x) \geq  \Vert \varphi_{i} \Vert_{u} + \int_{0}^{t} h_{i} (s) \left( \liminf\limits_{||x|| \rightarrow \infty} u_{i}(s,x) \right)^{p_{ii}} \left( \liminf\limits_{||x|| \rightarrow \infty} u_{j} (s,x)\right) ^{p_{ij}} ds.\label{ecpellim}
\end{equation}

Otherwise, taking the uniform norm in (\ref{ecintexp}) and using (\ref{intp}) we obtain
\begin{equation}
||u_{i} (t)||_{u}\leq ||\varphi_{i}||_{u} + \int_{0}^{t} h_{i}(s) ||u_{i}(s)||_{u}^{p_{ii}} ||u_{j}(s)||_{u}^{p_{ij}} ds.\label{ecpnorma}
\end{equation}

Introducing the auxiliary functions  
\begin{equation*}
w_{i} (t) = \liminf\limits_{\vert x \vert \rightarrow \infty} u_{i}(t,x), \;\;\; \tilde{w}_{i}(t) =||u_{i} (t)||_{u},
\end{equation*}
the inequalities (\ref{ecpellim}) and (\ref{ecpnorma}) can be written, for each $t\in[0,\tau_{(\ref{eedif})})$, as
\begin{align*}
w_{i} (t) \geq ||\varphi_{i}||_{u}+ \int_{0}^{t} h_{i} (s) w_{i}^{p_{ii}}(s) w_{j}^{p_{ij}}(s) ds, &
\\
\tilde{w}_{i} (t) \leq ||\varphi_{i}||_{u} + \int_{0}^{t} h_{i} (s) \tilde{w}_{i}^{p_{ii}} (s) \tilde{w}_{j}^{p_{ij}}(s) ds.
\end{align*}

The comparison Theorem \ref{lcpo} implies
\begin{equation*}
\tilde{w}_{i} (t) \leq y_{i}(t) \leq w_{i}(t),\ \ \text{for all } \ t \in [0,\tau_{(\ref{eedif})}),
\end{equation*}
then $ \tau_{(\ref{edoa})} = \tau_{(\ref{eedif})}$. Moreover, the above inequality also implies
\begin{equation*}
w_{i} (t) = \liminf\limits_{||x|| \rightarrow \infty} u_{i}(t,x) \leq \limsup\limits_{||x|| \rightarrow \infty} u_{i}(t,x) \leq ||u_{i} (t)||_{u} = \tilde{w}_{i} (t) \leq w_{i} (t).
\end{equation*}
Hence it follows that $y_{i}(t)=\lim_{||x|| \rightarrow \infty} u_{i}(t,x)$.\hfill
\end{proof} 

\section{Blow-up in finite time at space infinite}\label{SBlow}
In this section we will see that there is blow-up in finite time, $||u_{1}(\tau_{(\ref{eedif})}-)||_{u}+||u_{2}(\tau_{(\ref{eedif})}-)||_{u}$$=\infty$, but the blow-up is just at space infinite, $|u_{i}(\tau_{(\ref{eedif})}-,x)|<\infty$, for each $x\in\mathbb{R}^{n}$.

\begin{lemma}\label{respreetf}
Let $(u_{1},u_{2})$ be the solution of (\ref{eedif}) in $ [0,\tau_{(\ref{eedif})}) \times \mathbb{R}^{n}$. Suppose that for each $x_{0}\in \mathbb{R}^{n}$ and $\rho_{0}>0$ there are $t_{0}=t_{0}(x_{0},\rho_{0}) \in (\tau_{(\ref{eedif})}/2,\tau_{(\ref{eedif})})$ and $\theta=\theta(x_{0},\rho_{0}) \in (0,1) $ such that 
\begin{equation} \label{hpac} 
u_{i}(t,x) \leq \theta y_{i}(t), \;\; \text{for all } (t,x) \in [t_{0},\tau_{(\ref{eedif})}) \times B (x_{0},\rho_{0}).
\end{equation}
Then
\begin{equation*} 
u_{i}(t,x) \leq \theta y_{i} (\tau_{(\ref{eedif})} - \varepsilon/2), \;\; \text{for all } (t,x) \in [t_{0},\tau_{(\ref{eedif})}) \times B (x_{0},\rho_{0}/2),
\end{equation*}
for some constant $0<\varepsilon=\varepsilon(\rho_{0}) <\tau_{(\ref{eedif})}/2$.  
\end{lemma}

\begin{proof}
Let $(y_{1},y_{2})$ be the solution of the system (\ref{edoa}).  Extend the domain of $ y_{i} $ to $ (-\tau_{(\ref{eedif})},\tau_{(\ref{eedif})}) $ such that the extended function, also denoted by $ y_{i} $, is smooth. Set
\begin{equation*} 
g(r) = \cos^{2} \left( \dfrac{\pi r}{2\rho_{0}} \right), \;\; 0 \leq r < \dfrac{\rho_{0}}{2}.
\end{equation*}
Take $ 0 < \varepsilon < \tau_{(\ref{eedif})} $ and define
\begin{equation*} 
z_{i} (t,x) = \theta y_{i} (t-\varepsilon g(r)), \;\; (t,x) \in [t_{0},\tau_{(\ref{eedif})}] \times B(x_{0},\rho_{0}/2),
\end{equation*}
where $ r = \vert x-x_{0} \vert $. From (\ref{edoa}) we have
\begin{align*}
&\dfrac{\partial}{\partial t} z_{i} - \mathit{k}_{i}(t) \Delta z_{i} - h_{i}(t) z_{i}^{p_{ii}} z_{j}^{p_{ij}} = \theta y'_{i} (t- \varepsilon g(r)) \\
&\times\left\lbrace 1- \theta^{p_{ii}+p_{ij}-1}  - \varepsilon^{2} \mathit{k}_{i} (t) (g'(r))^{2} \, \dfrac{y''_{i}(t- \varepsilon g(r))}{y'_{i}(t- \varepsilon g(r))} + \varepsilon \mathit{k}_{i} (t) g''(r) + \frac{n-1}{r} \varepsilon \mathit{k}_{i} (t) g'_{i}(r) \right\rbrace.
\end{align*}

Observe that (using $|\sin(x)|\leq |x|$, $x\in \mathbb{R}$)
\begin{equation} 
\cos^{2} \left( \frac{\pi}{4} \right) \leq g(r) \leq 1, \;\; \vert g'(r) \vert \leq \frac{r}{2}\left(\frac{\pi}{\rho_{0}}\right)^{2}, \;\; 
\vert g''(r) \vert \leq \frac{1}{2} \left(\frac{\pi}{\rho_{0}} \right)^{2}, \;\;  r \in \left[0,\frac{\rho_{0}}{2}\right].  \label{estpag}
\end{equation}
Otherwise, using (\ref{edoa}) we obtain
\begin{equation} 
\dfrac{y''_{i}(t)}{y'_{i}(t)} = \frac{h'_{i}(t)}{h_{i}(t)} + p_{ii}h_{i}(t)y_{i}^{p_{ii}-1}(t)y_{j}^{p_{ij}}(t) + p_{ij}h_{j}(t)y_{i}^{p_{ji}}(t)y_{j}^{p_{jj}-1}(t).\label{fghdgh}
\end{equation}
Taking $0<\varepsilon<\tau_{(\ref{eedif})}/2<t_{0}<\tau_{(\ref{eedif})}$ we have
\begin{equation*} 
t - \varepsilon \leq t- \varepsilon g(r) < \tau_{(\ref{eedif})} - \varepsilon \cos^{2} \left( \frac{\pi}{4} \right),\ \ \text{for all } t\in [t_{0},\tau_{(\ref{eedif})}).
\end{equation*}
For each $t\in [t_{0},\tau_{(\ref{eedif})})$, the equality (\ref{fghdgh}), implies
\begin{align*} 
&\left| \dfrac{y''_{i}(t- \varepsilon g(r))}{y'_{i}(t- \varepsilon g(r))} \right| \leq \\
& \max\limits_{t \in \left[\frac{\tau_{(\ref{eedif})}}{2}, \tau_{(\ref{eedif})} - \varepsilon \cos^{2} \left( \frac{\pi}{4} \right) \right]} \left\lbrace \dfrac{\vert h'_{i}(t) \vert}{h_{i}(t)} + p_{ii}h_{i}(t)y_{i}^{p_{ii}-1}(t)y_{j}^{p_{ij}}(t) + p_{ij}h_{j}(t)y_{i}^{p_{ji}}(t)y_{j}^{p_{jj}-1}(t)\right\rbrace := K.
\end{align*}
The above estimation and (\ref{estpag}) lead us to
\begin{align*} 
&1 - \theta^{p_{ii}+p_{ij}-1} - \varepsilon^{2} \mathit{k}_{i} (t) \left( g'(r) \right)^{2} \dfrac{y''_{i}(t- \varepsilon g(r))}{y'_{i}(t- \varepsilon g(r))} + \varepsilon \mathit{k}_{i}(t) g''(r) + \frac{n-1}{r} \varepsilon \mathit{k}_{i}(t) g'(r) \\
&\geq  1 - \theta^{p_{ii}+p_{ij}-1} - \varepsilon \max\limits_{t \in \left[\frac{\tau_{(\ref{eedif})}}{2},\tau_{(\ref{eedif})}\right]} \mathit{k}_{i} (t)\left\lbrace \varepsilon \left( \frac{\pi}{\rho_{0}} \right)^{2} K + \frac{1}{2} \left( \frac{\pi}{\rho_{0}} \right)^{2} + \frac{n-1}{2} \left( \frac{\pi}{\rho_{0}} \right)^{3}  \right\rbrace.
\end{align*}
Since $p_{ii}+p_{ij}> 1$ and $\theta \in (0,1)$, then $1-\theta^{p_{ii}+p_{ij}-1}>0$. This implies that we can take $ \varepsilon=\varepsilon(\rho_{0})>0$ small enough for which
\begin{equation*} 
\dfrac{\partial}{\partial t} z_{i} (t,x) \geq \mathit{k}_{i} (t) \Delta z_{i} (t,x) + h_{i}(t) z_{i}^{p_{ii}} (t,x) z_{j}^{p_{ij}} (t,x).
\end{equation*}

Moreover, from (\ref{hpac}) we deduce the boundary conditions
\begin{align*} 
z_{i}(t,x) &\geq u_{i}(t,x),\ \  \text{for all } (t,x) \in [t_{0},\tau_{(\ref{eedif})}) \times \partial B (x_{0},\rho_{0}/2),\\
z_{i}(t_{0},x) &\geq u_{i}(t_{0},x),\ \  \text{for all } x \in B (x_{0},\rho_{0}/2).
\end{align*}
The comparison principle yields
\begin{equation*} 
u_{i}(t,x) \leq \theta y_{i}(t- \varepsilon g(r)),\ \ \text{for all } (t,x) \in [t_{0},\tau_{(\ref{eedif})}) \times B (x_{0},\rho_{0}/2).
\end{equation*}
Using that $ y_{i} $ is increasing and $ g $ is decreasing we get
\begin{equation*} 
u_{i}(t,x) \leq \theta y_{i} \left( t - \varepsilon \cos^{2} \left( \frac{\pi}{4} \right) \right), \ \ \text{for all } (t,x) \in [t_{0},\tau_{(\ref{eedif})}) \times B (x_{0},\rho_{0}/2).
\end{equation*}
From this the result follows easily. \hfill
\end{proof}

\bigskip
\begin{proof}[Proof of Theorem \ref{Thbenalin}]
We follow the ideas given in \cite{S-U}. Let $ x_{0} \in \mathbb{R}^{n} $ and $ r_{0} > 0 $ be fix and arbitrary. The strong maximum principle (see Theorem 1 in Chapter 2 of \cite{A-F}) implies that the solution $ (u_{1},u_{2}) $ of (\ref{eedif}) satisfies 
\begin{equation*}
u_{i} (t,x) < y_{i}(t), \;\; \text{for all } (t,x) \in (0,\tau_{(\ref{eedif})}) \times G,
\end{equation*}
for any compact set $ G \subset \mathbb{R}^{n}$. By a translation of time we can consider that the system (\ref{eedif}) begins at time $ t_{0} \in (0,\tau_{(\ref{eedif})}) $, then
\begin{equation*}
u_{i} (t_{0},x) < y_{i}(t_{0}), \;\; \text{for all } x \in \bar{B} (x_{0},r_{0}),
\end{equation*}
where $\bar{B} (x_{0},r_{0})=\{x \in \mathbb{R}^{n}:||x-x_{0}||\leq x_{0}\}$. Restarting the system at $t_{0}$ we may assume that $ t_{0} = 0 $. Therefore
\begin{equation*}
u_{i} (0,x) < y_{i}(0) = \Vert \varphi_{i} \Vert_{u}, \;\; \text{for all } x \in \bar{B} (x_{0},r_{0}).
\end{equation*}

Let $ \mathit{k}(t) := \mathit{k}_{1}(t) + \mathit{k}_{2}(t)$ and $ w (t,x) $ be the solution of (see \cite{A-F}) 
\begin{equation*}
\begin{array}{ll}
\dfrac{\partial w}{\partial t} = \mathit{k} (t) \Delta w, & x \in B(x_{0},r_{0}), t > 0, \\
w(t,x) = 1, & x \in \partial B(x_{0},r_{0}), t \geq 0,\\
w(0,x) \not\equiv 1, & x \in B(x_{0},r_{0}),\\
\max \left\lbrace \dfrac{u_{1}(0,x)}{y_{1}(0)}, \, \dfrac{u_{2}(0,x)}{y_{2}(0)} \right\rbrace \leq w(0,x) \leq 1, & x \in B(x_{0},r_{0}),\\
\Delta w(0,x) \geq 0, & x \in B(x_{0},r_{0}).
\end{array}
\end{equation*}

We are going to prove that $ (y_{1}w,y_{2}w) $ is a super solution of (\ref{eedif}), we means that
\begin{equation} \label{ssw}
y_{i} (t) w(t,x) \geq u_{i}(t,x), \;\;\;\; \text{for all } (t,x) \in (0,\tau_{(\ref{eedif})}) \times B (x_{0},r_{0}).
\end{equation}

The maximal principle implies (see Chapter 2 of \cite{A-F}),
\begin{equation} \label{deimpmp}
\Delta w(t,x) \geq 0, \;\; \text{for all } (t,x) \in (0,\infty) \times B (x_{0},r_{0}).
\end{equation}

Using the hypothesis $p_{ii}+p_{ij}> 1$, (\ref{edoa}), (\ref{deimpmp}) and $ w(t,x) \leq 1 $ we can conclude that
\begin{equation*}
\begin{split}
\dfrac{\partial}{\partial t} (y_{i}w) & = y'_{i} w + y_{i}\dfrac{\partial}{\partial t} w \\
& = h_{i}(t)y_{i}^{p_{ii}}(t)y_{j}^{p_{ij}}(t) w + y_{i}(t) \mathit{k}(t)\Delta w \\
& \geq h_{i}(t) \left( y_{i}(t) w \right)^{p_{ii}} \left( y_{j}(t) w \right)^{p_{ij}} + \mathit{k}_{i}(t)\Delta (y_{i}w).
\end{split}
\end{equation*}
Moreover, we see that
\begin{equation*} 
y_{i}(t) w(t,x) = y_{i}(t) \geq u_{i}(t,x), \;\; \text{for all } (t,x) \in (0,\tau_{(\ref{eedif})}) \times \partial B (x_{0},r_{0})
\end{equation*}
and
\begin{equation*} 
y_{i}(0) w(0,x) \geq y_{i}(0) \dfrac{u_{i}(0,x)}{y_{i}(0)} = u_{i}(0,x), \;\; \text{for all } x \in B (x_{0},r_{0}),
\end{equation*}
then we are able to apply the comparison principle to deduce (\ref{ssw}).

On the other hand, the strong maximum principle implies
\begin{equation} \label{epw}
0 \leq w(t,x) < 1, \;\; \text{for all } (t,x) \in (0,\infty) \times B (x_{0},r_{0}).
\end{equation}
For each $\tilde{r}_{0} \in (0,r_{0}) $ and $ t_{0} \in (\tau_{(\ref{eedif})}/2,\tau_{(\ref{eedif})}) $ we set
\begin{equation*} 
\theta =\theta(\tilde{r}_{0},t_{0}):= \sup \lbrace w(t,x): (t,x) \in [t_{0},\tau_{(\ref{eedif})}] \times \bar{B} (x_{0},\tilde{r}_{0}) \rbrace.
\end{equation*}
By (\ref{epw}) we have $ \theta < 1 $ and
\begin{equation} 
u_{i} (t,x) \leq w(t,x)y_{i}(t)\leq \theta y_{i}(t), \;\; \text{for all } (t,x) \in [t_{0},\tau_{(\ref{eedif})}) \times \bar{B} (x_{0},\tilde{r}_{0}).\label{aplcomp}
\end{equation}

For each $R>0$ let us consider the compact ball $\bar{B}(0,R)$. For each $x\in \bar{B} (0,R)$ the inequality (\ref{aplcomp}) implies (with $x_{0}=x$, $r_{0}=R$, $\tilde{r}_{0}=R/2=\rho_{0}$) that we can use Lemma \ref{respreetf}. Then there exist $t_{x} \in (\tau_{(\ref{eedif})}/2,\tau_{(\ref{eedif})})$ and $\theta_{x} \in (0,1) $ such that 
\begin{equation*} 
u_{i}(t,x) \leq \theta_{x} y_{i} (\tau_{(\ref{eedif})} - \varepsilon_{R}/2), \;\; \text{for all } (t,x) \in [t_{x},\tau_{(\ref{eedif})}) \times B (x,R/4),
\end{equation*}
for some $\varepsilon_{R}>0$. The family $\{B (x,R/4):x\in \bar{B}(0,R/4)\}$ is an open cover of $\bar{B}(0,R)$, then it admits a finite subcover $\{B (x_{1},R/4),...,B (x_{k},R/4)\}$. Let us take
$$\theta_{R}=\max\{\theta_{x_{1}},...,\theta_{x_{k}}\},\ \  t_{R}=\max\{t_{x_{1}},...,t_{x_{k}}\},$$
with such selection we obtain 
\begin{equation*} 
u_{i}(t,x) \leq \theta_{R} y_{i} (\tau_{(\ref{eedif})} - \varepsilon_{R}/2), \;\; \text{for all } (t,x) \in [t_{R},\tau_{(\ref{eedif})}) \times \bar{B}(0,R).
\end{equation*}
Then $\sup_{||x||\leq R} u_{i}(t,x)\leq \theta_{R} y_{i} (\tau_{(\ref{eedif})} - \varepsilon_{R}/2)$, for all $t\in [t_{R},\tau_{(\ref{eedif})})$. We conclude the proof taking $\limsup$ on $t$.\hfill
\end{proof}

\section{Bounds for the maximal existence time}\label{Sbounds}

In this section we recall a generalized version of Osgood's lemma and we will use this to obtain some bounds for the maximal existence time for the solutions of system (\ref{eedif}). We begin with the following comparison result. 

\begin{lemma}
Let $ (y_{1},y_{2}) $ be the positive solution of (\ref{edoa}) defined on $ [0,\tau_{(\ref{edoa})}) $. Let us assume that $ (\tilde{h_{i}})' \leq 0 $, for $i\in\{1,2\}$.
\begin{itemize}
\item[$(a)$] If  $a_{i} > 0 $ and $ \; a_{j} \neq 0$, then there exists a constant $c_{(\ref{epi})} > 0$ such that 
\begin{equation} \label{epi}
y_{i}(t) \geq c_{(\ref{epi})} (\tilde{h}_{i})^{1/a_{i}} (t) y_{j}^{a_{j}/a_{i}} (t), \;\; \text{for all } t \in [0,\tau_{(\ref{edoa})}).
\end{equation}
\item[$(b)$] If  $a_{i} > 0 $ and $ \; a_{j} = 0$, then there exists a constant $ c_{(\ref{epii})} > 0 $ such that 
\begin{equation} \label{epii}
y_{i}^{a_{i}}(t) \geq c_{(\ref{epii})} \tilde{h}_{i} (t) \log (y_{j}(t)), \;\; \text{for all } t \in [0,\tau_{(\ref{edoa})}).
\end{equation}
\item[$(c)$] If  $a_{i} = 0 $ and $ \; a_{j} = 0$, then
\begin{equation} \label{epiii}
y_{i}(t)\geq y_{i}(0)\left(\dfrac{y_{j}(t)}{y_{j}(0)}\right)^{\tilde{h}_{i}(t)}, \;\; \text{for all } t \in [0,\tau_{(\ref{edoa})}).
\end{equation}
\end{itemize}
\end{lemma}

\begin{proof}
$(a)$ Let us define the function  $J:[0,\tau_{(\ref{edoa})}) \rightarrow \mathbb{R}$ as
\begin{equation} \label{defji}
J(t) = M y_{i}^{a_{i}} (t) - \tilde{h}_{i}(t) y_{j}^{a_{j}} (t), 
\end{equation}
where $ M $ is a constant that we are going to fix later. Using (\ref{edoa}) we get
\begin{equation*}
J'(t) = h_{i}(t) y_{i}^{p_{ji}} (t) y_{j}^{p_{ij}} (t) \left[ M a_{i} - a_{j} \right] - (\tilde{h}_{i})'(t) y_{j}^{a_{j}} (t).
\end{equation*}
Otherwise the definition (\ref{defji}) yields
\begin{equation*}
J'(t) - \dfrac{(\tilde{h}_{i})'(t)}{\tilde{h}_{i}(t)} J(t) = h_{i}(t) y_{i}^{p_{ji}} (t) y_{j}^{p_{ij}} (t) \left[ M a_{i} - a_{j} \right] - M \dfrac{(\tilde{h}_{i})'(t)}{\tilde{h}_{i}(t)} y_{i}^{a_{i}} (t).
\end{equation*}

If we take $ M > \max \left\lbrace  0, a_{j}/a_{i} \right\rbrace $ then the hypothesis $(\tilde{h}_{i})' \leq 0 $ implies
\begin{equation*}
J'(t) - \dfrac{(\tilde{h}_{i})'(t)}{\tilde{h}_{i}(t)} J(t) \geq 0.
\end{equation*}
Multiplying the above inequality by 
$$ \exp \left\lbrace -\int_{0}^{t} \dfrac{(\tilde{h}_{i})'(s)}{\tilde{h}_{i}(s)} ds \right\rbrace,  $$ 
we have
\begin{eqnarray*}
\left(  J(t) \exp \left\lbrace -\int_{0}^{t} \dfrac{(\tilde{h}_{i})'(s)}{\tilde{h}_{i}(s)} ds \right\rbrace \right)' &=& \left( J'(t) - \dfrac{(\tilde{h}_{i})'(t)}{\tilde{h}_{i}(t)} J(t) \right)\\
&&\times \exp \left\lbrace -\int_{0}^{t} \dfrac{(\tilde{h}_{i})'(s)}{\tilde{h}_{i}(s)} ds \right\rbrace \geq 0.
\end{eqnarray*}
From this we obtain
\begin{equation*}
J(t) \geq J(0) \exp \left\lbrace -\int_{0}^{t} \dfrac{(\tilde{h}_{i})'(s)}{\tilde{h}_{i}(s)} ds \right\rbrace.
\end{equation*}
To prove the inequality (\ref{epi}) it is sufficient if we can take $J(0) \geq 0$. But from (\ref{defji}) we see that we get such inequality if we take
\begin{equation*}
M >\max \left\{\frac{a_{j}}{a_{i}}, \tilde{h}_{i}(0) \dfrac{y_{j}^{a_{j}} (0)}{y_{i}^{a_{i}}(0)}\right\}.
\end{equation*}

$(b)$ In this case we define the function  $J:[0,\tau_{(\ref{edoa})}) \rightarrow \mathbb{R}$ as
\begin{equation*}
J(t) = M y_{i}^{a_{i}} (t) - \tilde{h}_{i}(t)\log (y_{j}(t)).
\end{equation*}
Proceeding as before we have
\begin{equation*}
J'(t) - \dfrac{(\tilde{h}_{i})'(t)}{\tilde{h}_{i}(t)} J(t) = h_{i}(t) y_{i}^{p_{ji}} (t) y_{j}^{p_{ij}} (t) \left[ M a_{i} - 1 \right] - M \dfrac{(\tilde{h}_{i})'(t)}{\tilde{h}_{i}(t)} y_{i}^{a_{i}} (t).
\end{equation*}
The desired inequality (\ref{epii}) follows if we take 
\begin{equation*}
M >\max \left\{\frac{1}{a_{i}}, \log(y_{j}(0)) \dfrac{\tilde{h}_{i}(0)}{y_{i}^{a_{i}}(0)}\right\}.
\end{equation*}

$(c)$ From (\ref{edoa}) we see that $ y'_{i} (t) \geq 0$, then $ y_{i} (t) \geq y_{i} (0)>0 $. Let us define the function  $J:[0,\tau_{(\ref{edoa})}) \rightarrow \mathbb{R}$ as
\begin{equation*}
J(t) = \log \left( \dfrac{y_{i} (t)}{y_{i} (0)} \right) - \tilde{h}_{i}(t) \log \left( \dfrac{y_{j} (t)}{y_{j} (0)} \right).
\end{equation*}
Using that $ p_{ij} = p_{jj} -1$ and $p_{ji} = p_{ii} -1$, then the derivative of $J$ can be written as
\begin{equation*}
J'(t) - \dfrac{(\tilde{h}_{i})'(t)}{\tilde{h}_{i}(t)} J(t) = - \dfrac{(\tilde{h}_{i})'(t)}{\tilde{h}_{i}(t)} \log \left(\dfrac{y_{i} (t)}{y_{i} (0)} \right) \geq 0.
\end{equation*}
As before we can deduce  
\begin{equation*}
\log \left( \dfrac{y_{i} (t)}{y_{i} (0)} \right) \geq \tilde{h}_{i}(t) \log \left( \dfrac{y_{j} (t)}{y_{j} (0)} \right), \;\; \text{for all } t \in [0,\tau_{(\ref{edoa})}).
\end{equation*}
From which the inequality (\ref{epiii}) is deduced readily. \hfill
\end{proof}

\bigskip
To state the generalized version of Osgood's lemma we introduce some nomenclature. For $y_{0}>0$ and $b,f:(0,\infty) \rightarrow (0,\infty)$ continuous fun\-cti\-ons let us define 
\begin{equation*}
B(x) = \int_{y_{0}}^{x} \dfrac{ds}{b(s)} \ \ \text{and} \ \  F(x) = \int_{0}^{x} f(s)ds.
\end{equation*}

\begin{lemma}
The solution of the ordinary differential equation
\begin{equation}
y'(t) =  f(t)b(y(t)),\ \ y(0) = y_{0},\label{eqosgood}
\end{equation}
with $ y_{0} > 0 $, is
\begin{equation*}
y(t) = B^{-1} (F(t)).
\end{equation*}
The domain of $ y $ is $ [0,F^{-1}(B(\infty)) $ if $ B(\infty) < F(\infty) $, or $ [0,\infty) $ if $ B(\infty) \geq F(\infty) $. The time $ \tau_{(\ref{eqosgood})} = F^{-1} (B(\infty)) $ is the blow-up time.
\end{lemma}

\begin{proof}
An elementary proof can be seen in \cite{V-M}. \hfill
\end{proof}

\begin{proposition}\label{wecandance}
Let $ (y_{1},y_{2}) $ be the positive solution of (\ref{edoa}) defined on $ [0,\tau_{(\ref{edoa})}) $. If for some $i,j\in\{1,2\}$, $ p_{ii} > 1$ and 
\begin{equation}
\int_{0}^{\infty} h_{i}(s) ds > \frac{1}{p_{ii}-1}\cdot\dfrac{y_{i}^{1-p_{ii}}(0)}{y_{j}^{p_{ij}}(0)}.\label{indpropo}
\end{equation}
then $\tau_{(\ref{edoa})}<\infty$.
\end{proposition}

\begin{proof}
From (\ref{edoa}) we see that $y_{j}$ is increasing ($y_{j}'\geq 0$), then $ y_{j} (t) \geq y_{j} (0) $. Using this in (\ref{edoa}) yields
\begin{equation*}
y_{i}' (t) \geq h_{i} (t) y_{i}^{p_{ii}} (t) \left( y_{j}(0) \right)^{p_{ij}}.
\end{equation*}
Now let us consider the ordinary differential equation
\begin{equation}\label{epze}
z_{i}'(t) = \left( y_{j}(0) \right)^{p_{ij}} h_{i} (t) z_{i}^{p_{ii}} (t), \ \ z_{i}(0)=y_{i}(0).
\end{equation}
The comparison Theorem \ref{lcpo} implies that $ y_{i} (t) \geq z_{i} (t) $. On the other hand, using the notation on the Osgood's lemma we have
$$B_{i}(\infty)=\int_{y_{i}(0)}^{\infty}\frac{ds}{s^{p_{ii}}}=\frac{y_{i}^{1-p_{ii}}(0)}{p_{ii}-1}$$
and 
$$F_{i}(\infty)=y_{j}^{p_{ij}}(0)\int_{0}^{\infty}h_{i}(s)ds.$$
The hypothesis (\ref{indpropo}) implies $F_{i}(\infty)>B_{i}(\infty)$, then Osgood's lemma brings about $\tau_{(\ref{edoa})}\leq \tau_{(\ref{edoa}),i}\leq\tau_{(\ref{epze}),i}\leq F_{i}^{-1}(B_{i}(\infty))<\infty.$\hfill
\end{proof}

\bigskip
\begin{proof}[Proof of Theorem \ref{thest}] $(a.1)$ From the system (\ref{edoa}) and the estimation (\ref{epi}) we get
\begin{align*}
y'_{j} (t) & = h_{j}(t)y_{j}^{p_{jj}}(t)y_{i}^{p_{ji}}(t)\\
& \geq h_{j}(t)y_{j}^{p_{jj}}(t) \left( c_{(\ref{epi})} \tilde{h}_{i}^{1/a_{i}} (t) y_{j}^{a_{j}/a_{i}} (t) \right)^{p_{ji}} \\
& = f_{1,j} (t) y_{j}^{\alpha_{j}} (t),
\end{align*}
where
\begin{equation*}
f_{1,j}(t) = c_{(\ref{epi})}^{p_{ji}} h_{j} (t) \tilde{h}_{i}^{p_{ji}/a_{i}} (t).
\end{equation*}
Let us consider the associated ordinary differential equation
\begin{equation} \label{e1j}
z_{j}'(t) = f_{1,j}(t)z_{j}^{\alpha_{j}} (t),\ \ z_{j}(0) = y_{j} (0).
\end{equation}
By the comparison Theorem \ref{lcpo}  we have, $ y_{j} \geq z_{j} $. Therefore $ \tau_{(\ref{edoa}),j} \leq \tau_{(\ref{e1j}),j}$, then the result follows from Osgood's lemma.\\
$(a.2)$ The estimation (\ref{epi}) turns out
\begin{equation*}
y_{j} (t) \leq c_{(\ref{epi})}^{-a_{i}/a_{j}} \tilde{h}_{i}^{-1/a_{j}}(t) y_{i}^{a_{i}/a_{j}} (t),
\end{equation*}
and using the system (\ref{edoa}) we obtain
\begin{equation*}
y'_{i} (t) \leq f_{1,i} (t) y_{i}^{\alpha_{i}} (t),
\end{equation*}
where
\begin{equation*}
f_{1,i} (t) = c_{(\ref{epi})}^{-p_{ij}a_{i}/a_{j}} h_{i}(t) \tilde{h}_{i}^{-p_{ij}/a_{j}}(t).
\end{equation*}
The corresponding differential equation is
\begin{equation} \label{e1i}
z_{i}'(t) = f_{1,i}(t)z_{i}^{\alpha_{i}} (t),\ \ z_{i}(0) = y_{i} (0).
\end{equation}
By the comparison Theorem \ref{lcpo}  we have $\tau_{(\ref{e1i}),i} \leq \tau_{(\ref{edoa}),i}$. In this case we would like to have $\tau_{(\ref{e1i}),i} = \infty $ and this is ensured in the following cases.

If $ \alpha_{i} = 1 $, then
\begin{equation*}
z_{i} (t) = y_{i}(0) e^{F_{1,i}(t)}, \;\;\; t \in [0,\infty).
\end{equation*}

If $ \alpha_{i} < 1 $, then
\begin{equation*}
z_{i} (t) = \left( y_{i}^{1-\alpha_{i}} (0) + (1-\alpha_{i}) F_{1,i}(t) \right)^{1/(1-\alpha_{i})}, \;\;\; t \in [0,\infty).
\end{equation*}

If $ \alpha_{i} > 1 $ and $ F_{1,i} (\infty) \leq y_{i}^{1-\alpha_{i}} (0) / (\alpha_{i}-1) $, then the result is deduced using the Osgood's lemma, as before.\\
$(b.1)$ Due to $y_{j}(t)\geq y_{j}(0)> 1$, then (\ref{edoa}) and (\ref{epii}) implies
\begin{equation*}
y'_{j} (t) \geq f_{2,j} (t) \varphi_{2,j} (y(t)),
\end{equation*}
where
\begin{equation*}
f_{2,j}(t) = c_{(\ref{epii})}^{p_{ji}} h_{j} (t) \tilde{h}_{i}^{p_{ji}/a_{i}} (t), \;\; \varphi_{2,j} (x) = x^{p_{jj}}\left( \log x\right)^{p_{ji}/a_{i}}.
\end{equation*}
Considering the differential equation
\begin{equation} \label{e2j}
z'_{j} (t) =   f_{2,j}(t) \varphi_{2,j} (z_{j}(t)), \ \ z_{j} (0) =  y_{j} (0), 
\end{equation}
we have, by the comparison Theorem \ref{lcpo}, $ y_{j} \geq z_{j} $. Hence $ \tau_{(\ref{edoa}),j} \leq \tau_{(\ref{e2j})} $. The result follows form Osgood's lemma. \\
$(b.2)$ The inequality (\ref{epii}) turns out
\begin{equation*}
y_{j} (t) \leq \exp \left\lbrace  c_{(\ref{epii})}^{-1} (\tilde{c}_{i})^{-1} y_{i}^{a_{i}} (t) \right\rbrace,
\end{equation*}
therefore from (\ref{edoa}) we deduce
\begin{equation*}
y'_{i} (t) \leq h_{i}(t) \varphi_{2,i} (y_{i}(t)),
\end{equation*}
where 
\begin{equation*}
\varphi_{2,i} (x) = x^{p_{ii}} \exp \left\lbrace p_{ij} c_{(\ref{epii})}^{-1} (\tilde{c}_{i})^{-1} x^{a_{i}} \right\rbrace.
\end{equation*}
Considering the equation
\begin{equation} \label{e2i}
z_{i}'(t) = h_{i}(t)\varphi_{2,i} (z_{i}(t)),\ \ z_{i}(0) = y_{i} (0),
\end{equation}
the comparison Theorem \ref{lcpo} implies $ y_{i} \leq z_{i} $, then $ \tau_{(\ref{e2i})} \leq \tau_{(\ref{edoa}),i}$. Again, the result is con\-se\-quence of Osgood's lemma.\\
$(c.1)$ From (\ref{edoa}) and (\ref{epiii}) we have
\begin{equation*}
y'_{j} (t) \geq f_{3,j}(t) y_{j}^{\beta_{j}} (t),
\end{equation*}
where
\begin{equation*}
f_{3,j}(t) =\left( \dfrac{y_{i}(0)}{y_{j}^{\tilde{c}_{i}} (0)} \right)^{p_{ij}} h_{j} (t).
\end{equation*}
To get the result we proceed as we did in the case {\bf (a.1)}.\\
$(c.2)$ Using (\ref{epiii}) we get
\begin{equation*}
y_{j} (t) \leq y_{j} (0) \left( \dfrac{y_{i}(t)}{y_{i}(0)} \right)^{1/\tilde{c}_{i}},
\end{equation*}
then (\ref{edoa}) implies
\begin{equation*}
y'_{i} (t) \leq f_{3,i}(t) y_{i}^{\gamma_{i}} (t),
\end{equation*}
where 
\begin{equation*}
f_{3,i}(t) =\left( \dfrac{y_{j}(0)}{y_{i}^{1/ \tilde{c}_{i}} (0)} \right)^{p_{ij}} h_{i} (t).
\end{equation*}
Now proceeding as we did in the case $(a.2)$ we get the desired result.\hfill
\end{proof}

\bigskip
\begin{proof}[Proof of Corollary \ref{Corollary}]
The case $ p_{11} > 1 $ or $ p_{22} > 1 $ follows from Proposition \ref{wecandance} since 
\begin{equation*}
\int_{0}^{\infty} h_{i}(s) ds = \infty.
\end{equation*}
From Theorem \ref{thest} $(a.1)$ we see that $\tau_{(\ref{edoa})} < \infty $ if
\begin{equation}\label{edpe}
\left( a_{1} > 0,\ \alpha_{2} > 1 \right) \;\; \text{or} \;\; \left( a_{2} > 0,\ \alpha_{1} > 1 \right). 
\end{equation}
Therefore it is enough to see that (\ref{edpe}) is equivalent to
\begin{equation}\label{fdpx}
\left( p_{11}-1 \right) \left( p_{22}-1 \right) - p_{12}p_{21} < 0.
\end{equation}

Let us suppose that
\begin{equation*}
p_{21} - p_{11} + 1 > 0, \;\; p_{22}+p_{21}\cdot \dfrac{p_{12} - p_{22} + 1}{p_{21}-p_{11}+1} > 1,
\end{equation*}
then
\begin{equation*}
p_{21}p_{12} - p_{22}p_{11} + p_{22} + p_{11} > 1,
\end{equation*}
and this implies (\ref{fdpx}). Analogously we see that $ a_{2} > 0 $, $ \alpha_{1} > 1 $ implies (\ref{fdpx}).

To see the reciprocal statement notice that $ a_{1} \leq 0 $ and $ a_{2} \leq 0 $, implies
\begin{equation*}
p_{11}-1 \geq p_{21} \;\; \text{and} \;\; p_{22}-1 \geq p_{12}
\end{equation*}
then
\begin{equation*}
\left( p_{11}-1 \right) \left( p_{22}-1 \right) \geq p_{21}p_{12},
\end{equation*}
which is a contradiction to (\ref{fdpx}). Therefore $ a_{1} > 0 $ or $ a_{2} > 0 $, from this and the inequality (\ref{fdpx}) we get (\ref{edpe}). \hfill 
\end{proof}

\subsection*{Acknowledgment}

This work was partially supported by the grant PIM20-1 of Universidad Aut\'onoma de Aguascalientes.

\end{document}